\documentclass[reqno,12pt,letterpaper]{amsart}
\usepackage{amsmath,amssymb,amsthm,graphicx,mathrsfs,url}
\usepackage[usenames,dvipsnames]{color}
\usepackage[colorlinks=true,linkcolor=Red,citecolor=Green]{hyperref}

\def\?[#1]{\textbf{[#1]}\marginpar{\Large{\textbf{??}}}}

\newtheorem{theo}{Theorem}
\newtheorem{prop}{Proposition}[section]

\newtheorem{lemm}[prop]{Lemma}
\newtheorem{corr}[prop]{Corollary}
\newtheorem{rem}{Remark}
\numberwithin{equation}{section}

\newcommand{\mc}{\mathcal}
\newcommand{\rr}{\mathbb{R}}
\newcommand{\nn}{\mathbb{N}}
\newcommand{\cc}{\mathbb{C}}

\newcommand{\la}{\lambda}
\newcommand{\eps}{\epsilon}

\newcommand{\pl}{\partial}
\newcommand{\x}{\times}

\newcommand{\til}{\widetilde}

\newcommand{\cjd}{\rangle}
\newcommand{\cjg}{\langle}

\newcommand{\demi}{\tfrac{1}{2}}

\title[Horocyclic invariance of Ruelle resonant states]{Horocyclic invariance of Ruelle resonant states for contact Anosov flows in dimension $3$}
\author{Colin Guillarmou}
\email{colin.guillarmou@math.u-psud.fr}
\address{Laboratoire de Math\'ematiques d'Orsay, Univ. Paris-Sud, CNRS, Universit\'e Paris-Saclay, 91405 Orsay, France}
\author{Fr\'ed\'eric Faure}
\email{frederic.faure@univ-grenoble-alpes.fr}
\address{Institut Fourier, UMR 5582, 100 rue des Maths, BP74 38402 St Martin d?H\`eres.}

\begin{document}
\maketitle

\begin{abstract}
We show that for contact Anosov flows in dimension $3$ the resonant states associated to the first band of Ruelle resonances are distributions that are invariant by the unstable horocyclic flow.  
\end{abstract}

\section{Introduction}
Since the work of Butterley-Liverani \cite{BuLi} and Faure-Sj\"ostrand \cite{FaSj}, one can define an intrinsic discrete spectrum for the vector field $X$ generating a smooth Anosov flow on a compact manifold $\mc{M}$. More precisely, one view $P:=-X$ as a first order differential operator and we can construct appropriate anisotropic Sobolev spaces $\mc{H}^N$ (depending on parameter $N>0$) related to the stable/unstable splitting of the flow, on which the first order differential operator $P-\la$ is an analytic family of Fredholm operators of index $0$ in the complex half-plane $\{{\rm Re}(\la)>C_0-\mu N\}$ for some $C_0\geq 0$  and $\mu>0$ depending on $X$; here $N>0$ can be taken as large as we like. The eigenvalues and the eigenstates of $P$ are independent of $N$, they are called \emph{resonances} and 
\emph{resonant states}.  The operator is not self-adjoint on $\mc{H}^N$ and there can be Jordan blocks. We say that $u\in \mc{H}^N$ is a \emph{generalized resonant state} with resonance $\la_0\in \{{\rm Re}(\la)>C_0-\mu N\}$ if $(P-\la_0)^ju=0$ for some $j\in \nn$. 
An equivalent way to define resonances for $P$ is through the resolvent: the resolvent $R_{P}(\la):=(P-\la)^{-1}$ is an analytic family of bounded operators on $L^2(\mc{M},dm)$ (for some fixed Lebesgue type measure $dm$) in $\{{\rm Re}(\la)>C_0\}$ for some $C_0\geq 0$, there exists a meromorphic continuation of $R_P(\la)$ to $\la \in \cc$ as a map 
\[ R_P(\la): C^{\infty}(\mc{M})\to \mc{D}'(\mc{M})\] 
and the polar part of the Laurent expansion of $R_P(\la)$ at a pole $\la_0$ is a finite rank operator. The resonances are the poles of $R_P(\la)$ and the generalized resonant states are the elements in the range of the residue
\[\Pi_{\la_0}:=-{\rm Res}_{\la_0}R_P(\la)\]
which turns out to be a projector. 

We will now assume that $\mc{M}$ is a closed oriented manifold with dimension $3$ and that $X$ generates a contact Anosov flow, i.e there is a smooth one-form $\alpha$ such that $d\alpha$ is symplectic on $\ker \alpha$, $\alpha(X)=1$ and $i_Xd\alpha=0$. 
We fix a smooth metric $G$ on $\mc{M}$ and we denote by $E_s$ and $E_u$ the stable and unstable bundles, the tangent bundle has a flow-invariant continuous splitting 
\begin{equation}\label{split} T\mc{M}= \rr X\oplus E_s\oplus E_u
\end{equation} 
such that there is $C>1$ and $\eps>0$ such that for all $z\in \mc{M}$, there is 
$\mu_{\max}(z)>\mu_{\min}(z)>\eps$ so that 
\begin{equation} \label{estimeehyp-intro}
\begin{gathered}
\forall \xi\in  E_s(z), \forall t\geq 0, \,\,\,\,
C^{-1}e^{-\mu_{\max}(z)t}|\xi|_G\leq |d\varphi_t(z).\xi|_G\leq Ce^{-\mu_{\min}(z)t}|\xi|_G,\\
\forall \xi\in  E_u(z), \forall t\geq 0, \,\,\,\,
C^{-1}e^{-\mu_{\max}(z)t}|\xi|_G\leq  |d\varphi_{-t}(z).\xi|_G\leq Ce^{-\mu_{\min}(z) t}|\xi|_G.
\end{gathered}\end{equation}
We define the mimimal/maximal expansion rates of the flow
\begin{equation}\label{numin-intro}
\begin{gathered}
 \mu_{\max}:=\lim_{t\to +\infty}\sup_{z\in \mc{M}}-\frac{1}{t}
 \log \Big|d\varphi_t(z)|_{E_s(z)}\Big|_G=\lim_{t\to +\infty}\sup_{z\in \mc{M}}-\frac{1}{t}
 \log \Big|d\varphi_{-t}(z)|_{E_u(z)}\Big|_G, \\
 \mu_{\min}:=\lim_{t\to +\infty}\inf_{z\in \mc{M}}-\frac{1}{t}
 \log \Big|d\varphi_t(z)|_{E_s(z)}\Big|_G=\lim_{t\to +\infty}\inf_{z\in \mc{M}}-\frac{1}{t}
 \log \Big|d\varphi_{-t}(z)|_{E_u(z)}\Big|_G.
\end{gathered}\end{equation}
We assume that $E_u$ is an orientable bundle and let $U_-$ be a global non-vanishing section of $E_u$, called an \emph{unstable horocyclic vector field}. 
By Hurder-Katok \cite{HuKa}, $U_-$ is a vector field that can be chosen with regularity 
$C^{2-\eps}(\mc{M})$ for all $\eps>0$. For a contact Anosov flow, there is a preserved smooth measure $dm:=\alpha\wedge d\alpha$, thus $P$ is skew-adjoint on $L^2(\mc{M},dm)$ and $R_P(\la)$ is analytic in ${\rm Re}(\la)>0$ (the $L^2$-spectrum is the whole imaginary line). The operator $U_-$ can be viewed as acting on the negative Sobolev space $H^{-s}(\mc{M})$ for $s<1$ as follows: for $u\in H^{-s}(\mc{M})$, for all $f\in C^\infty(\mc{M})$, 
\[\cjg U_-u,f\cjd:=\cjg u,-U_-f-{\rm div}(U_-)f\cjd\]
where ${\rm div}(U_-)$ is the divergence of $U_-$ with respect to $dm$. 
\begin{theo}\label{main}
Let $\mc{M}$ be a smooth 3-dimensional oriented compact manifold and let $X$ be a smooth vector field generating a contact Anosov flow. Assume that the unstable bundle is orientable. For $P=-X$,  if $\la_0$ is a resonance of $P$ with ${\rm Re}(\la_0)>-\mu_{\min}$ and if $u$ is a generalized resonant state of $P$ with resonance $\la_0$, then $U_-u=0$. 
\end{theo}
In view of the regularity of the stable/unstable foliation in our case, we have locally near each point $x_0\in \mc{M}$ a decomposition of $\mc{M}$ as a product 
$W_u\x W_s\x (-\eps,\eps)_t$ using the stable/unstable foliation, where $W_{u/s}$ are diffeomorphic to $(-\eps,\eps)$. The flow is $X=\pl_t$ is those coordinates, and Theorem \ref{main} says that a resonant state $w$ with resonance $\la_0$ (if ${\rm Re}(\la_0)>-\mu_{\min}$) is of the form 
\[ w(u,s,t)=e^{-\la_0t}\omega(s)\]
for some distribution $\omega$ on $W_s$, i.e the resonant state depends in a non-trivial way only on the variable $s$ of the stable leaves. In fact, due to the wave-front set analysis of resonant states in \cite{FaSj}, a resonant state $w$ can be restricted locally to each piece of local stable leaf (which is an embedded smooth submanifold), or alternatively the lift of $w$  to the universal cover $\til{\mc{M}}$ of $\mc{M}$ can be restricted to the stable leaves $\til{\mc{M}}$.

The horocyclic invariance of the first band of resonances was shown in constant cuvature by Dyatlov-Faure-Guillarmou \cite{DFG}, and follows also for hyperbolic surfaces from the work of Flaminio-Forni \cite{FlFo}. It is quite stricking that this type of properties still holds for variable curvature cases. The first resonant state for a certain transfer operator associated to an Anosov diffeomorphism on $\mathbb{T}^2$ is also shown to be horocyclic invariant by Giuletti-Liverani \cite{GiLi}. There are other related cases which appeared in the work of Dyatlov \cite{Dy} for resonances of semi-classical operators with $r$-normally hyperbolic trapped set, but the resonant states are only microlocally killed by some smooth pseudo-differential operator playing the role of $U_-$.

In Theorem \ref{mainresult}, we prove a more general result which applies to  the operator $P:=-X+V$ where 
$V$ is a regular potential, and where the unstable derivative $U_-$ is replaced by $U_-+\alpha_V$ for some appropriate function $\alpha_V$ depending on $V$. The operator $U_-+\alpha_V$ can be viewed as a covariant derivative in the unstable direction. Interesting particular cases are for $V=r_-$, where resonant states are in $\ker U_-^*$, and
for $V=\demi r_-$, the case studied intensively by Faure-Tsujii \cite{FaTs2}; see Corollary 
\ref{maincor}. 

Using the work of Faure-Tsujii \cite{FaTs}, we deduce the following result about existence of an infinite dimensional space of horocyclic invariant distributions: 
\begin{corr}\label{cormain}
Let $\mc{M}$ be a smooth $3$-dimensional oriented manifold and let $X$ be a smooth vector field generating a contact Anosov flow. Assume that $E_u$ is orientable and that $\mu_{\max}<2\mu_{\min}$. Then, for each $\eps>0$ small, there exist infinitely many resonant states in $\ker U_-$ with associated resonances contained in the band 
\[\{{\rm Re}(\la)\in [-\demi \mu_{\max}-\eps,
-\demi\mu_{\min}+\eps]\}.\] These resonant states belongs to the Sobolev space $H^{-\tfrac{1}{2}\frac{\mu_{\max}+2\eps}{\mu_{\min}}}(\mc{M})$.
\end{corr}

The proof of Theorem \ref{main} follows the strategy of \cite{DFG} for hyperbolic surfaces. Let us briefly explain the idea. 
 A resonant state $u$ for $-X$ with resonance $\la_0\in \cc$ satisfies $(-X-\la_0)u=0$ where $u$ is a distribution whose microlocal singularities (wave-front set) are 
contained in the subbundle$E_u^*\subset T^*\mc{M}$ defined by the condition $E_u^*(\rr X\oplus E_u)=0$. 
Applying $U_-$ to the equation $(-X-\la_0)u=0$, we get $(-X-\la_0-r_-)U_-u=0$ and one can show that $\omega:=U_-u$ also has its main microlocal singularities at $E_u^*$ by using the regularity $r_-\in C^{2-\eps}(\mc{M})$, for all $\eps>0$.
Now if $\la_0$ belongs to the spectral region where $-X-r_-$ has no eigenstates with microlocal singularities in $E_u^*$, we can conclude that $\omega=0$. We prove that this condition is satisfied when ${\rm Re}(\la_0)>-1=-\mu_{\min}$.

We actually provide two proofs in the paper. The first one, contained in Theorem \ref{main}, uses resolvent identities: we show that $U_-$ intertwines the resolvent of $-X$ with that of $-X-r_-$. The second proof is more technical, uses microlocal methods and follows the argument just described above.

Applying the results of Corollary \ref{maincor} with $V=r_-$, we also get that resonant states for $-X+r_-$ are obstructions to solving the cohomological equation $U_-f=g$ with $g\in C^{1}(\mc{M})$ for the unstable vector field $U_-$, in the spirit of the work of Flaminio-Forni \cite{FlFo} in constant curvature; see the discussion in Section \ref{lastsec}.

We notice that our proof would apply similarly in higher dimension under pinching conditions on the Lyapunov condition, except that one needs to use a covariant derivative in the unstable direction. The horocyclic invariance of resonant states apply only to finitely many resonant states, for there is only finitely many resonance in the complex region where our result would hold, by a result of Tsujii \cite{Ts}. We have thus decided to focus only on the case of dimension $3$, where in addition the regularity of $E_u$ is known to be better.\\

\textbf{Acknowledgements.} C.G. is supported by ERC consolidator grant IPFLOW. F.F. and C.G. are supported by the grant ANR 13-BS01-0007-01. We thank T. Alazard, S. Crovisier, S. Dyatlov, S. Gou\"ezel, B. Hasselblatt, C. Liverani and T. De Poyferr\'e for useful discussions and references.

\section{Stable/unstable bundles}

\subsection{Anosov flows and the regularity of stable/unstable bundles} 
Let $\mc{M}$ be a smooth compact $3$-dimensional oriented manifold and let $X$ be an  vector field, with flow denoted by $\varphi_t$ that is Anosov. 
We fix a smooth metric $G$ on $\mc{M}$ and 
we denote by $E_s$ and $E_u$ the stable and unstable bundles so that one has the 
flow-invariant continuous splitting \eqref{split} with \eqref{estimeehyp-intro}.
Let $\alpha$ be the continuous flow-invariant $1$-form on $\mc{M}$ so that $\ker \alpha=E_u\oplus E_s$ and $\alpha(X)=1$. By Hurder-Katok \cite[Theorem 2.3]{HuKa}, if  $\alpha\in C^1(\mc{M};T^*\mc{M})$ then $\alpha\in C^\infty(\mc{M};T^*\mc{M})$ and either 
$\alpha\wedge d\alpha=0$ or it is a nowhere vanishing $3$-form and $\varphi_t$ is a \emph{contact flow}: $i_Xd\alpha=0$ and $d\alpha$ is symplectic on $\ker \alpha$.
We shall assume in what follows that we are in case of a contact flow. 
In that case, since the symplectic form $d\alpha$ on $\ker \alpha$ is preserved by $\varphi_t$, we can find $\nu_{\min/\max}(z), \mu_{\min/\max}(z)$ such that
\[\nu_{\min}(z)=\mu_{\min}(z), \quad \nu_{\max}(z)=\mu_{\max}(z).\]
Let us also define the dual Anosov decomposition 
\[ T^*\mc{M}=\rr \alpha\oplus E^*_u\oplus E_s^*, \quad \textrm{ with }E_s^*(E_s\oplus \rr X)=0,\quad  E_u^*(E_u\oplus \rr X)=0.\]
In \cite{HuKa}, Hurder-Katok proved the following regularity statement on the unstable/stable bundles.
\begin{lemm}[Hurder-Katok]\label{Hurder-Katok}
For a smooth contact flow in dimension $3$, the regularity of the bundles $E_u$ and $E_s$ is 
\begin{equation}
\forall r<2, \,\,  E_u \in C^{r}  , \quad \forall r<2,\,\, E_s\in C^{r}.
\end{equation}
\end{lemm}
By regularity $C^r$ of a bundle, it is meant that the bundle is locally spanned by vector fields 
which have $C^r$ coefficients in smooth charts on $\mc{M}$. For what follows, we will write $f\in C^{2-}(\mc{M})$ to mean that a function/vector field belongs to $\cap_{\delta>0}C^{2-\delta}(\mc{M})$.

Anosov \cite{An} proved  that  there exist local stable and unstable smooth submanifolds $W_s(z),W_u(z)$ of $\mc{M}$ at each point $z$,  whose dependence in $z$ is only H\"older and such that $T_zW_u(z)=E_u(z)$ and $T_zW_s(z)=E_s(z)$. The submanifolds $W_u$ form a foliation near $z$ and from \cite[Lemma 3.1]{DMM}, there are continuous maps 
\[\Lambda: V_1\x V_2\to M , \quad V_1\subset \rr ,V_2\subset \rr^2 \]
such that $\Lambda_y: V_1\to M$ defined by $\Lambda_y(x)=\Lambda(x,y)$ is a $C^\infty$
embedding with image an unstable local submanifold $W_u(z)$ for some $z$ and the derivatives $\pl_x^\beta \Lambda $ are continuous on $V_1\x V_2$ for all $\beta\in \nn$.
The same holds for the stable foliation.

Next, we want to make sense of unstable derivatives. 
\begin{lemm}\label{existunstable}
Assume $X$ generates a smooth contact flow on an orientable $3$-dimensional manifold 
$\mc{M}$ and that $E_u$ is an orientable bundle. 
There exists a non-vanishing vector field $U_-$ on $\mc{M}$ with regularity $C^{2-}(\mc{M};T\mc{M})$ such that $U_-(z)\in E_u(z)$ for all $z\in\mc{M}$, and there exists a function $r_-$ with regularity $C^{2-}(\mc{M})$ such that
\begin{equation}\label{XU_-}
[X,U_-]=-r_-U_-.
\end{equation} 
The function $r_-$ satisfies $d\varphi_{-t}(z).U_-(z)=e^{-\int_{-t}^{0} r_-(\varphi_s(z)) ds}U_-(\varphi_{-t}(z))$ for $t\geq 0$ and
\begin{equation}\label{estimeejacuns} 
Ce^{-t\mu_{\max}(z)}\leq e^{-\int_{-t}^0 r_-(\varphi_s(z))ds}\leq Ce^{-t\mu_{\min}(z)}.
\end{equation} 
If $a_i$ are the coefficients of $U_-$ in a smooth coordinates system, then $U_-^k(a_i)$ are continuous for all $k\in \nn$.
The same properties hold with $U_+$ replacing $U_-$, $E_s$ replacing $E_u$, $r_+$ replacing $r_-$, with $d\varphi_{t}(z).U_+(z)=e^{-\int_{0}^{t} r_+(\varphi_s(z)) ds}U_+(\varphi_{t}(z))$, and $U_+$ is a $C^{2-}(\mc{M};T\mc{M})$ section of $E_s$ with local coefficients $b_i$ such that $U_+^k(b_i)$ are continuous for all $k\in \nn$.
\end{lemm}
\begin{proof}
The orientability of $E_u$ insures that there exists a non-vanishing vector field $U$ which is a section of $E_u$, and we normalize it so that its $G$-norm is $||U||_G=1$. It can be chosen to be globally $C^{2-}(\mc{M})$ by Lemma \ref{Hurder-Katok}. By the remark following the Lemma (which describes the unstable foliation regularity), we also have that the coefficients $a_i$ of $U$ in local coordinates are such that $U^n(a_i)$ are continuous for all $n\in\nn$. 
We approximate $U$ by a smooth vector field $U_\eps$ in a way that $|U-U_\eps|_G\leq \eps$ for $\eps>0$ small. Since $\mc{M}$ is oriented and $3$-dimensional (thus parallelizable), 
we can find a smooth vector field $S$ so that $(X,U_\eps,S)$ is a global smooth basis of $T\mc{M}$, and we write $U=a_\eps U_\eps+bX+cS$ with $|a_\eps-1|=\mc{O}(\eps)$ and $a_\eps,b,c\in C^{2-}(\mc{M})$. Let us define $U_-:=(1/a_\eps)U$ which is also a $C^{2-}(\mc{M})$ non-vanishing section of $E_u$ for $\eps>0$ fixed small enough.
Since $d\varphi_t(z).E_u(z)=E_u(\varphi_t(z))$, we have $d\varphi_t(z).U_-(z)=f(t,z)U_-(\varphi_t(z))$ for some $f(t,z)\in C^{2-}(\rr\x \mc{M})$ with $f(t,z)>0$, and 
$\pl_tf(t,z)\in C^{1-}(\rr\x\mc{M})$\footnote{It is probably known from experts that $\pl_tf(t,z)\in C^{2-}(\rr \x \mc{M})$, from which  $r_-\in C^{2-}(\mc{M})$ would follow, but we haven't found references for such a fact, which is the reason why we use the approximation argument involving $U_\eps$.}. We also have $f(s+t,z)=f(s,z)f(t,\varphi_s(z))$.
We differentiate at $t=0$ and get \eqref{XU_-} with $r_-(z):=\pl_tf(0,z)/f(0,z)$ and  more generally $\pl_sf(s,z)/f(s,z)=r_-(\varphi_s(z))$. A priori $r_-\in C^{1-}(\mc{M})$
but a small computation using $[X,U_-]=-r_-U_-$ implies that 
\[-r_- =h+ck \]
where $h,k\in C^\infty(\mc{M})$ are the $U_\eps$ components of $[X,U_\eps]$ and $[X,S]$ in the basis $(X,U_\eps,S)$. Thus $r_-\in C^{2-}(\mc{M})$. The regularity of the coefficients 
of $U_-$ when differentiated twice in the direction $U_-$ follows from the same property as for  $U$. 
By definition of $r_-$ we also have that
\[  |d\varphi_{-t}(z)U_-(z)| = e^{-\int_{-t}^0r_-(\varphi_s(z))ds}\]
and this completes the proof.
\end{proof}

\begin{rem}
We notice that $U_\pm$ are not uniquely defined: one can always multiply $U_\pm$ by a positive smooth function $f$, and $fU_\pm$ would satisfy all the same properties as $U_\pm$ described in Lemma \ref{existunstable}. On the other hand, the kernel of $U_-$ is independent of the choice of non-vanishing section $U_-$ of $E_u$.
\end{rem} 

It is interesting to give the following interpretation to \eqref{XU_-}, which explains why the operator $P=-X-r_-$ appears naturally: the flow acts on the bundle $E_s^*$, and if $\omega$ is a non-vanishing section of $E_s^*$ defined by $\omega|_{E_s\oplus X}=0$ and $\omega(U_-)=1$, we have $\mc{L}_X\omega=r_-\omega$; thus  
for each $f\in C^2(\mc{M})$,
\[ \mc{L}_{-X}(f\omega)=(-Xf-r_-f)\omega.\]
The map $\pi: C^{2-}(\mc{M}; E_s^*)\to C^{2^-}(\mc{M})$ defined by $\pi(h):=h(U_-)$ is an isomorphism with inverse $e:C^{2^-}(\mc{M})\to C^{2-}(\mc{M}; E_s^*)$ given by $e(f)=f\omega$, one has $\pi \mc{L}_{-X}e=-X-r_-$
and \eqref{XU_-} can be reinterpreted as the identity: for each $f\in C^{\infty}(\mc{M})$
\[     \mc{L}_{-X} d^uf=d^u\mc{L}_{-X}f \]
where $d^u: C^{\infty}(\mc{M})\to C^{2-}(\mc{M},E_s^*)$ is the operator defined by $d^uf:=df|_{E_u}$. We refer to \cite[Section 3.3.2]{FaTs2} for a related discussion.\\

To conclude this section, we define the minimal and maximal expansion rates by 
\begin{equation}\label{muminmax} 
\mu_{\min}:= \lim_{t\to +\infty}\inf_{z\in\mc{M}} \frac{1}{t}\int_{0}^t r_-(\varphi_s(z))ds, \quad 
\mu_{\max}:= \lim_{t\to +\infty}\sup_{z\ni\mc{M}} \frac{1}{t}\int_{0}^t r_-(\varphi_s(z))ds.
\end{equation}
First, we remark that the two limits as exist as $t\to +\infty$ by Fekete's lemma 
since $F_1(t):=\sup_{z\in\mc{M}} \int_{0}^t r_-(\varphi_s(z))ds$ is easily seen to be a subadditive function and $F_2(t):=\inf_{z\ni\mc{M}} \int_{0}^t r_-(\varphi_s(z))ds
$ is superadditive. By Lemma \eqref{existunstable}, for each $\eps>0$, there is $C_\eps$ 
such that for all $t\geq 0$ and all $z\in \mc{M}$
\begin{equation}\label{boundwithmu}
C_\eps^{-1}e^{-t(\mu_{\max}+\eps)}\leq \Big|d\varphi_{-t}(z)|_{E_u}\Big|_G\leq C_\eps e^{-t(\mu_{\min}-\eps)} \end{equation}

\subsection{The case of geodesic flow} 
To illustrate the discussion above, let us discuss the special case of the geodesic flow of negatively curved surfaces.
Let $(M,g)$ be a smooth oriented compact Riemannian surface with Gauss curvature $K(x)<0$ and let 
$SM$ be its unit tangent bundle with the projection $\pi_0:SM\to M$. We define $\mc{M}=SM$ and the geodesic flow at time $t\in\rr$ is denoted by 
$\varphi_t:SM\to SM$, its generating vector field is denoted by $X$ as above. 
The generator of rotations $R_s(x,v):=(x,e^{is}v)$ in the fibers of $SM$ is a smooth vertical vector field denoted by $V$. 
Let $X_\perp:=[X,V]$, this is a horizontal vector field and $(X,X_\perp,V)$ is an orthonormal basis for the Sasaki metric $G$ on $SM$. We have the commutator formulas (see for example \cite{PSU})
\begin{equation}\label{commut}
[X,X_\perp]= -KV, \quad [V,X_\perp]=X.
\end{equation}

The Jacobi equation along a geodesic $x(t)=\pi_0(\varphi_t(x,v))$ is 
\begin{equation}\label{Jacobi} 
\ddot{y}(t)+K(x(t))y(t)=0.\end{equation}
For $(x,v)\in SM$ and $a,b\in \rr$, one has 
\begin{equation}\label{dvarphi}
d\varphi_t(x,v).(-aX_\perp+bV)=-y(t)X_\perp(\varphi_t(x,v))+\dot{y}(t)V(\varphi_t(x,v))
\end{equation}
if $y(t)$ solves the Jacobi equation with $y(0)=a,\dot{y}(0)=b$. 
Notice that the function $r(t)=\dot{y}(t)/y(t)$ solves the Riccati equation
\begin{equation} \label{riccati}
\dot{r}(t)+r(t)^2+K(x(t))=0
\end{equation}
for the times so that $y(t)\not=0$. 
For $T\in \rr$, let $y_T(t,x,v)$ be the solution of the Jacobi equation \eqref{Jacobi}  along the geodesic 
$x(t)=\pi_0(\varphi_t(x,v))$ with conditions
\[y_T(0,x,v)=1,  \quad y_T(T,x,v)=0.\]
Since $g$ has no conjugate points, $y_T(t,x,v)\not=0$ when $t\not=T$.
Let $r_T(t,x,v):=\dot{y}_T(t,x,v)/y_T(t,x,v)$ which solves \eqref{riccati}, it is defined for $t<T$ and $r_T(t,x,v)\to -\infty$ as $t\to T$. By Hopf \cite{Ho}, the following limits exist for all 
$t,x,v$
\[ r_+(t,x,v):=-\lim_{T\to +\infty}r_T(t,x,v), \quad r_-(x,v):=\lim_{T\to +\infty}r_{-T}(t,x,v).\]
We denote $r_\pm(x,v):=r_\pm(0,x,v)$ and we see that $r_\pm(t,x,v)=r_\pm(\varphi_t(x,t))$.
We have $r_\pm>0$  and they solve the Riccati equation on $SM$
\begin{equation}\label{riccati2} 
\mp Xr_\pm+r_\pm^2+K=0.
\end{equation}
The functions $r_\pm(x,v)$ are smooth in the $X$ direction and are globally H\"older.
We define the vector fields  
\[ U_-:=X_\perp -r_-V ,\quad U_+:=X_\perp +r_+V \]
\begin{lemm}\label{Upm}
The following commutation relations hold
\[ [X,U_-]=-r_-U_-, \quad [X,U_+]=r_+U_+,\]
the function $r_\pm$ are in $C^{2-}(\mc{M})$ and   
\[ d\varphi_t(x,v).U_-(x,v)=e^{\int_{0}^tr_-(\varphi_s(x,v))ds}U_-(\varphi_t(x,v)),\]
\[ d\varphi_t(x,v).U_+(x,v)=e^{-\int_{0}^tr_+(\varphi_s(x,v))ds}U_+(\varphi_t(x,v)).\]
\end{lemm}
\begin{proof}
We just compute, using \eqref{commut}  and the fact that $r_\pm$ solves \eqref{riccati2},
\[ [X,X_\perp-r_-V]=-KV-X(r_-)V-r_-X_\perp=-r_-(X_\perp-r_-V)=-r_-U_-\]
and similarly for $[X,U_+]$. By \eqref{dvarphi}, we have for each $(x,v)\in SM$
\[ d\varphi_t(x,v).U_-=-y(t)X_\perp+\dot{y}(t)V\]
where $\ddot{y}+Ky=0$ and $y(0)=-1$ and $\dot{y}(0)=-r_-(x,v)$. Clearly we have 
$w:=\dot{y}/y$ which satisfies the Riccati equation \eqref{riccati} with $w(0)=r_-(x,v)$, thus 
$w(t)=r_-(\varphi_t(x,v))$. This implies 
\[ d\varphi_t(x,v).U_-(x,v)=-y(t)U_-(\varphi_t(x,v)).\]
and $y(t)=-e^{\int_{0}^tr_-(\varphi_s(x,v))ds}$, and it shows $U_\pm$ are sections of $E_u$ and $E_s$. By Lemma \ref{Hurder-Katok}, since $X_\perp,V$ is a smooth frame, we deduce that $r_\pm$ are in $C^{2-}(\mc{M})$.
\end{proof}
We remark that by Klingenberg \cite{Kl}, if the Gauss curvature satisfies $-k_0^2\leq K(x)\leq -k_1^2$
 for some $k_0>k_1>0$, then there exists $C>0$ (depending only on $k_0/k_1$) so that for each $z\in SM$ 
\[ \forall \xi\in E_s(z),\forall t\geq 0,  \,\,  Ce^{-k_0t}|\xi|_G\leq |d\varphi_t(z).\xi|_G \leq Ce^{-
k_1t}|\xi|_G.\]
In particular this implies the bounds 
\begin{equation}\label{boundsonexp}
k_0\geq \mu_{\max}(z)\geq \mu_{\min}(z)\geq k_1.
\end{equation}

\section{Resonant states and horocyclic invariance}

\subsection{Analytic preliminaries}
We first recall basic facts about microlocal analysis.
Let $dm:=\alpha\wedge d\alpha$ be the contact measure  on $\mc{M}$ associated to the contact form $\alpha$, that is invariant by the flow. 
We use the notation $H^{s}(\mc{M})$ for the $L^2$-based Sobolev space (with respect to $dm$) of order 
$s\in \rr$, the space $C^{\gamma}(\mc{M})$ denotes the Banach space of H\"older functions with order $\gamma\in \rr^+\setminus \nn$; for $k\in \nn_0$ we shall write $C^k(\mc{M})$ for the space of functions $k$-times differentiable and with continuous $k$-derivatives. We will write $(C^{\gamma}(\mc{M}))'$ for their dual spaces and $C^{\gamma-}(\mc{M})=\cap_{\eps>0}C^{\gamma-\eps}(\mc{M})$.
We recall the embedding (see \cite[Chapter 7.9]{H"o}) 
\begin{equation}\label{embeddinghol}
\textrm{ if }\gamma \not\in \nn, \,\, C^\gamma(\mc{M})\subset H^{s}(\mc{M}) \textrm{ for }s<\gamma, \quad \textrm{ if }k\in\nn_0, \, C^k(\mc{M})\subset H^{k}(\mc{M}).
\end{equation}
We denote by $\Psi^{s}(\mc{M})$ the space of pseudo-differential operators
of order $s\in \rr$ (see for example \cite[Chap. 7]{Ta2}), i.e. which have Schwartz kernel that can be written in local coordinates as 
\[ K(x,x')= \frac{1}{(2\pi)^3}\int_{\rr^3}e^{i(x-x')\xi}\sigma(x,\xi)d\xi\]
where $\sigma(x,\xi)$ is smooth and satisfies the following symbolic estimates of order $s$
\[ \forall \alpha,\beta \in \nn^3, \exists C_{\alpha,\beta}>0, \quad |\pl_x^{\alpha}\pl_\xi^\beta \sigma(x,\xi)|\leq 
C_{\alpha,\beta}\cjg \xi\cjd^{s-|\beta|}.\]
For $A\in \Psi^s(\mc{M})$, there is a homogeneous symbol $\sigma_p$ on $T^*\mc{M}$ of order $s$, called principal symbol, so that in local coordinates $\sigma-\sigma_p$ is a symbol of order $s-1$ outside $\xi=0$.  We say that $A$ is elliptic in a conic set $W\subset T^*\mc{M}$
if there is $C>0$ such that $|\sigma_p(x,\xi)|>C|\xi|^{s}$ in $W$ for $|\xi|>1$.
The wave-front set of a distribution $u\in \mc{D}'(\mc{M})$ is the closed conic subset ${\rm WF}(u)\subset T^*\mc{M}\setminus\{\xi=0\}$ defined by: $(x_0,\xi_0)\notin {\rm WF}(u)$ if and only if there is $A\in \Psi^0(\mc{M})$  elliptic in a conic open set $W$ containing $(x_0,\xi_0)$ such that $Au\in C^\infty(\mc{M})$.

We also need to define  spaces of pseudo-differential operators with limited smoothness.
Following Taylor \cite{Ta}, for $\gamma\geq 0$ and $\gamma+m\geq 0$, we denote by $C^{\gamma}S^m(\rr^3)$ the class of symbols $\sigma(x,\xi)$ compactly supported in $x$, such that for all $\alpha$ there is $C_\alpha>0$ such that 
\[ ||\pl_\xi^\alpha \sigma(\cdot,\xi)||_{C^{\gamma}}\leq C_\alpha \cjg \xi\cjd^{m-|\alpha|}.\]
We will write $C^{1-}S^m$, resp. $C^{2-}S^m$, for symbols that are in all $C^{\gamma}S^m$ spaces with $\gamma<1$, resp. $\gamma<2$. 
We will always denote by ${\rm Op}$ the left quantization of symbols on $\rr^3$, defined by 
\[ {\rm Op}(\sigma)f(x)=\frac{1}{(2\pi)^3}\int_{\rr^6}e^{i(x-x')\xi} \sigma(x,\xi)f(x')d\xi dx'.\]
A subclass of $C^{\gamma}S^m(\rr^3)$ that will be used is the class of classical symbols, denoted $C^{\gamma}S_{\rm cl}^m(\rr^3)$, defined by the extra condition 
\[\sigma(x,\xi)\sim \sum_{j=0}^\infty \sigma_j(x,\xi),\quad |\xi|\to \infty\]
with $\sigma_j(x,\xi)$ homogeneous of degree $m-j$ in $\xi$. 
\begin{lemm}\label{boundednessPSDO}
If $\gamma\notin \nn$ and $m\in\rr$ so that $\gamma+m>0$,  then for each $\sigma\in C^{\gamma}S_{\rm cl}^m(\rr^3)$ the following operator is bounded
\begin{equation}\label{boundTaylor} 
{\rm Op}(\sigma): C_c^{\gamma+m}(\rr^3)\to C_c^{\gamma}(\rr^3).
\end{equation}
If $\gamma\geq 0$, $m\in\rr$ and $s>\gamma$, the following operator is bounded
\begin{equation}\label{bounddednessHs} 
\begin{gathered}
\gamma\notin \nn_0, \,\, {\rm Op}(\sigma): H^{s+m}(\rr^3)\to H^{\gamma}(\rr^3), \\
\gamma\in\nn_0, \,\, {\rm Op}(\sigma): H^{\gamma+m}(\rr^3)\to H^{\gamma}(\rr^3)
\end{gathered}
\end{equation}
\end{lemm}
\begin{proof}
The bound \eqref{boundTaylor} is Proposition 1.A in \cite{Ta}. 
The proof of Proposition 1.1 in \cite{Ta} reduces to the case of a homogeneous 
symbol $\sigma_0$ of degree $0$ in $\xi$. Indeed,  one can write $\sigma=\sigma_r+
\sum_{j=1}^N\sigma_{j}$ for some $N\in\nn$ where $\sigma_r\in C^{\gamma}S^{m-N}(\rr^3)$ and $\sigma_j$ are homogeneous symbols of degree $m-j$ in $\xi$ and $C^{\gamma}$ in $x$.
The operator $\sigma_r$ has Schwartz kernel in $C^{\gamma}(\rr^3\x\rr^3)$ if $N$ is large enough and thus  the good boundedness properties. For the homogeneous symbol $\sigma_j$, one writes it as a converging sum 
\[ \sigma_j(x,\xi)=\sum_{\ell=0}^\infty p_{j\ell}(x){\rm Op}(|\xi|^{m-j}\omega_\ell(\xi/|\xi|))\]
where $\omega_\ell$ are the spherical harmonics on $S^2$. The $p_{j\ell}$ functions decay faster than any polynomials in $\ell$ in $C^{\gamma}$ norm, and ${\rm Op}(|\xi|^{m-j}\omega_\ell(\xi/|\xi|))$ maps $C_c^{\gamma+m}(\rr^3)$ to $C^\gamma(\rr^3)$ with norm growing polynomially in $\ell$. The same argument then shows \eqref{bounddednessHs} since ${\rm Op}(|\xi|^{m-j}\omega_\ell(\xi/|\xi|))$ maps $H^{s+m-j}(\rr^3)$ to $H^s(\rr^3)$ with norm $\mc{O}(1)$.
\end{proof}

\subsection{Discrete spectrum in Sobolev anisotropic spaces}

We recall the results of Butterley-Liverani \cite{BuLi} and Faure-Sj\"ostrand \cite{FaSj}.
\begin{prop}[Faure-Sj\"ostrand]\label{fauresjos}
Let $X$ be a smooth vector field generating an Anosov flow on a compact manifold $\mc{M}$, let $V\in C^\infty(\mc{M})$ and let $P=-X+V$ be the associated first-order differential operator.\\
1) There exists $C_0\geq 0$ such that the resolvent $R_{P}(\la):=(P-\la)^{-1}:L^2(\mc{M})\to L^2(\mc{M})$ of $P$ is defined for ${\rm Re}(\la)>C_0$ and extends meromorphically to $\la\in \cc$ as a family of bounded operators 
$R_P(\la): C^\infty(\mc{M})\to \mc{D}'(M)$. The poles are called Ruelle resonances, the operator $\Pi_{\la_0}:=-{\rm Res}_{\la_0}R_P(\la)$ at a pole $\la_0$ is a finite rank projector and there exists $p\geq 1$ such that  $(P-\la_0)^p\Pi_{\la_0}=0$. The distributions in ${\rm Ran}\, \Pi_{\la_0}$ are called generalized resonant states and those in ${\rm Ran}\, \Pi_{\la_0}\cap \ker(P-\la_0)$ are called resonant states.\\ 
2) For each $N\in [0,\infty)$, there exists a Sobolev space $\mc{H}^{N}$
so that $C^\infty(\mc{M})\subset \mc{H}^{N}\subset H^{-N}(\mc{M})$ and such that $R_P(\la):\mc{H}^N\to \mc{H}^N$ is a meromorphic family of bounded operators 
in ${\rm Re}(\la)>C_0-N\mu_{\min}$, and $(P-\la):{\rm Dom}(P)\cap \mc{H}^{N}\to \mc{H}^N$ is an analytic family of Fredholm operators\footnote{Here ${\rm Dom}(P):=\{u\in\mc{H}^N; Pu\in \mc{H^N}\}$ is the domain of $P$ equipped with the graph norm.} in that region with inverse given by $R_P(\la)$.\\
3) For each $N_0>0$ large enough (depending on $N$) and each conic neighborhood $W$ of $E_u^*$, $\mc{H}^N$ can be chosen in such a way that $H^{N_0}(\mc{M})\subset \mc{H}^N$, and for each $A\in \Psi^0(\mc{M})$ microsupported outside $W$, one has $Au\in H^{N_0}(\mc{M})$ for all $u\in\mc{H}^N$.  For a resonance $\la_0$, the wave-front set of each generalized resonant state $u\in {\rm Ran}(\Pi_{\la_0})$ is contained in $E_u^*$. 
\end{prop}
The space $\mc{H}^{N}$ is called an anisotropic Sobolev space. The statement in \cite{FaSj} is only for the case with no potential (i.e $V=0$), but their proof applies as well to the case $P=-X+V$ as long as $V\in C^{\infty}(\mc{M})$.
It also follows readily from the proof of \cite{FaSj} that, if the flow of $X$ perserves a smooth measure $dm$ and $V=0$, then one can take $C_0=0$. For a general potential and a flow preserving a smooth measure, we can give an estimate on $C_0$: let us define the quantity
\[ 
V_{\max}:=\lim_{t\to -\infty}\sup_{z\in \mc{M}} \frac{1}{|t|}\int_{t}^0V(\varphi_s(z))ds.\]
\begin{lemm}
Let $V\in C^\infty(\mc{M})$ and assume $X$ is a smooth vector field generating an Anosov
flow preserving a smooth measure $dm$. The resolvent $R_P(\la)$ of Proposition \ref{fauresjos} is analytic in $\la$ as an $L^2(\mc{M})$ bounded operator in 
${\rm Re}(\la)>V_{\max}$. For each $N>0$,  $R_P(\la) : \mc{H}^N\to H^{-N}(\mc{M})$ is a meromorphic  family of bounded  operators in the region ${\rm Re}(\la)>V_{\max}-N\mu_{\min}$.
\end{lemm}
\begin{proof}
The resolvent of $P=-X+V$ for ${\rm Re}(\la)\gg 1$ large enough is given by the expression 
\[ R_P(\la)f= -\int_{-\infty}^0 e^{\la t+\int_{t}^0V\circ \varphi_s\, ds}f\circ \varphi_t \, dt.\]
We see that it converges in $L^2(\mc{M},dm)$ in the region $\{{\rm Re}(\la)>V_{\max}\}$ by using first the estimate $||f\circ \varphi_t||_{L^2(dm)}=||f||_{L^2(dm)}$ and the pointwise bounds (following from Cauchy-Schwarz) 
\[|R_P(\la)f(z)|^2\leq C_{\la,\eps} \int_{-\infty}^0 e^{{\rm Re}(\la) t-t(V_{\max}+\eps)}|f(\varphi_t(z))|^2 \, dt\]
for some constant $C_{\la,\eps}$ depending on ${\rm Re}(\la), \eps>0$ 
and $\eps>0$ that can be chosen as small as we want. The second statement is a consequence of the radial point estimates proved in Dyatlov-Zworski \cite[Theorem E.56]{DyZw}: indeed, since for $f\in \mc{H}^N$  we know that $BR_P(\la)f\in H^{N_0}(\mc{M})$ for some large $N_0$
and $B\in \Psi^{0}(\mc{M})$ elliptic outside a small conic neighborhood of $E_u^*$, we can use 
\cite[Proposition E.53]{DyZw} and the fact that 
\[ \int_0^t V\circ \varphi_s\, ds -t{\rm Re}(\la)-N \log||d\varphi_t|_{E_u^*}||<0\]
for large $t>0$ to deduce that $R_P(\la)f\in H^{-N}(\mc{M})$.
\end{proof}

In \cite{BuLi}, Butterley-Liverani deal with non-smooth flows. Even though it is not explicitely written in their paper, their technique allows to deal with potentials $V\in C^{1+q}(\mc{M})$, $q\in(0,1)$. In fact, the analysis with potentials is done carefully by Gou\"ezel-Liverani \cite{GoLi} for  Anosov diffeomorphisms using the same technique. Combining the methods of \cite[Theorem 1]{BuLi} for flows with the arguments of \cite[Proposition 4.4. and Theorem 6.4.]{GoLi} (taking $p=1,q<1$ and $\iota=0$ in their notations, since our flow is $C^\infty(\mc{M})\subset C^{p+q+1}(\mc{M})$), one obtains:
\begin{prop}[Butterley-Liverani, Gou\"ezel-Liverani]\label{BGL}
Let $V\in C^{1+q}(\mc{M})$ for some $0<q<1$ and let $X$ be a smooth vector field generating an Anosov flow preserving a smooth measure $dm$ in dimension $3$. 
There exist a Banach space $\mc{B}_{1,q}$ satisfying: for each $q'>q$ one has $C^1(\mc{M})\subset \mc{B}_{1,q}\subset (C^{q'}(\mc{M}))'$, the operator $P=-X+V$ has discrete spectrum in the region ${\rm Re}(\la)>\rho- q\mu_{\min}$ and the resolvent 
$R_P(\la)=(P-\la)^{-1}:\mc{B}_{1,q}\to \mc{B}_{1,q}$ is meromorphic there. Here $\rho:={\rm Pr}(V-r_-)$ is the topological pressure of the potential $V-r_-$ and $r_-$ is the function of Lemma \ref{existunstable}.
\end{prop}
We notice that $\rho={\rm Pr}(V-r_-)\leq V_{\max}$ by using ${\rm Pr}(-r_-)=0$.

\subsection{Horocyclic invariance of resonant states for contact flows. Short proof}
In this section, we shall assume that $\mc{M}$ is a 3-dimensional oriented compact manifold and $X$ is a smooth vector field generating a contact Anosov flow, with oriented unstable bundle. Here $dm$ will denote the contact measure and $V\in C^1(\mc{M})$ a potential. Due to the $C^{2-}$ regularity of $U_-$,  for $u\in H^{-1+\eps}(\mc{M})$ we can define $\omega=U_-u$ as a distribution by the expression
\[ \forall f\in C^\infty(\mc{M}),\quad  \cjg U_-u,f\cjd:=\cjg u,-(U_-f+{\rm div}(U_-)f)\cjd;\]
here ${\rm div}$ denotes the divergence with respect to $dm$ and 
 $-U_--{\rm div}(U_-)$ is the adjoint to $U_-$ with respect to $dm$. 
The quantity ${\rm div}(U_-)$ is in $C^{1-}(\mc{M})$, thus if $u$ is a resonant state, $U_-u$ is well-defined as long as ${\rm Re}(\la)>-\mu_{\min}$ since $u\in \mc{H}^{1-\eps}\subset H^{-1+\eps}(\mc{M})$ for some $\eps>0$ in that case.

We define the transfer operator
\[ \mc{L}^t: C^\infty(\mc{M})\to C^\infty(\mc{M}), \quad (\mc{L}^tf)(x):=f(\varphi_t(x)).\]
It extends as a  bounded operator on $L^2(\mc{M},dm)$ with norm $||\mc{L}^t||_{L^2\to L^2}=1$. If $V\in C^1(\mc{M})$ and $P=-X+V$, we also define the operator 
\[ e^{-tP}: C^1(\mc{M})\to C^1(\mc{M}), \quad e^{-tP}f:=e^{-\int_0^t\mc{L}^sVds}\mc{L}^tf
\]
satisfying $\pl_t(e^{-tP}f)=-Pe^{-tP}f$. Let us first prove an easy Lemma. 
\begin{lemm}\label{normL^s}
For each $\eps>0$, there exists $C_\eps>0$ such that for each $s\in [-1,1]$ and each $t\in \rr$, 
the operator $\mc{L}^t$ is bounded on $C^1(\mc{M})$ with norm 
\begin{equation}\label{boundC1}
||\mc{L}^t||_{C^{1}\to C^{1}}\leq  C_\eps e^{(\mu_{\max}+\eps) |t|}
\end{equation}
and on $H^{s}(\mc{M})$ with norm
\begin{equation}\label{boundHalpha} 
 ||\mc{L}^t||_{H^{s}\to H^{s}}\leq C_\eps e^{|s|(\mu_{\max}+\eps) |t|}.
 \end{equation}
\end{lemm}
\begin{proof} The $C^1$ bound follows from the definition of $\mu_{\max}$. 
We have $||\mc{L}^t||_{L^2\to L^2}=1$ and for each $\eps>0$, there is $C_\eps>0$ 
such that for all $u\in C^\infty(\mc{M})$ and $x\in \mc{M}$
\[ |d\mc{L}^tu|_{G_x}=|du_{\varphi_t(x)}.d\varphi_t(x)|_{G_x}\leq C_\eps e^{(\mu_{\max}+\eps) |t|}|du(\varphi_t(x))|_{G_{\varphi_t(x)}}
\]
thus by integrating the square of this inequality on $\mc{M}$ and using that $\varphi_t$ preserves $dm$, we get $||d\mc{L}^tu||_{L^2}\leq C_\eps e^{(\mu_{\max}+\eps) |t|}||du||_{L^2}$ and 
$||\mc{L}^t||_{H^1\to H^1}\leq C_\eps e^{(\mu_{\max}+\eps) |t|}$. Interpolating between $H^1$ and $L^2$ we get the result for $s\geq 0$ and using that $(\mc{L}^t)^*=\mc{L}^{-t}$ we obtain the desired result for $s \leq 0$.
\end{proof}
As a direct corollary, we get 
\begin{corr}\label{cor1}
If ${\rm Re}(\la)>\mu_{\max}+V_{\max}$, the resolvent $R_P(\la)$ of $P=-X+V$ is bounded as a map 
\[ R_P(\la): C^1(\mc{M})\to C^1(\mc{M}).\]
\end{corr}
\begin{proof}
The resolvent of $P=-X+V$ for ${\rm Re}(\la)>0$ is given by the expression 
\[ R_P(\la)f= -\int_{-\infty}^0 e^{\la t+\int_{t}^0\mc{L}^sVds}\mc{L}^tf \, dt\]
and \eqref{boundC1} shows that the integral converges in $C^1$ norm if ${\rm Re}(\la)>\mu_{\max}+V_{\max}$. 
\end{proof}

Next, define the potential  $W:=V-r_-$ and the quantities 
\[ 
W_{\max}:=\lim_{t\to -\infty} \sup_{z\in \mc{M}}\frac{1}{|t|}\int_{t}^0W(\varphi_s(z))ds.\]
which in turn are bounded by $W_{\max}\leq V_{\max}-\mu_{\min}$. 
We obtain 
\begin{lemm}\label{resolvr}
Let $r_-$ be the function of Lemma \ref{existunstable}, $V\in C^1(\mc{M})$ and 
$W=V-r_{-}$.
The operator $P'=-X+W$ has an analytic resolvent $R_{P'}(\la): C^0(\mc{M})\to C^0(\mc{M})$ in the region $\{{\rm Re}(\la)>W_{\max}\}$, given by the convergent expression
\begin{equation}\label{formulaRPr}
R_{P'}(\la)f:= -\int_{-\infty}^0 e^{\la t+\int_{t}^0 \mc{L}^sWds}(\mc{L}^tf) dt
\end{equation}
and satisfying $(P'-\la)R_{P'}(\la)={\rm Id}$ in the distribution sense. If $f\in C^1(\mc{M})$, then for ${\rm Re}(\la)\geq W_{\max}+s\mu_{\max}$ 
with $s\in (0,1]$, we have for all $\eps>0$
\begin{equation}\label{regul} 
R_{P'}(\la)f \in C^{s-\eps}(\mc{M}).
\end{equation}
Finally, there is no $C^0(\mc{M})$ solution $\omega$ to $(P'-\la)\omega=0$ in the region $\{{\rm Re}(\la)>W_{\max}\}$.
\end{lemm}
\begin{proof} The proof of the first statement is straightforward using that for each $\eps>0$ small, we have for 
$t<0$ large enough and uniformly on $\mc{M}$
\[ \int_{t}^0 \mc{L}^sW \, ds \leq (W_{\max}+\eps)|t|.\]

For the regularity \eqref{regul}, we observe that for $s=1$ this follows directly from the expression \eqref{formulaRPr} and the bound \eqref{boundC1}. To obtain the $s<1$ case, it suffices to use interpolation (i.e Hadamard three line theorem) between the line 
${\rm Re}(\la)=W_{\max}+\eps$ where we have $C^0$ bounds 
and the line ${\rm Re}(\la)=W_{\max}+\mu_{\max}$ where we have $C^1$ bounds.

To prove that $(P'-\la)$ is injective on $C^0(\mc{M})$, assume $(P'-\la)\omega=0$ and let $\omega(t)=\mc{L}^t\omega\in C^0(\mc{M})$. We have in the weak sense
\[ \pl_t \omega(t)=\mc{L}^tX\omega=-\mc{L}^t(r_--V+\la)\omega=(\mc{L}^t(W) -\la) \omega(t)\]
and therefore $\omega(t)=\omega e^{-\la t-\int_{t}^0 \mc{L}^sW\,ds}$. Since 
$||\omega(t)||_{C^0}\leq ||\omega||_{C^0}$, we can let $t\to -\infty$ and we obtain a contradiction if $\omega\not=0$.
\end{proof}

A first consequence of Lemma \eqref{resolvr} is that for each $V\in C^1(\mc{M})$ there exists a function $\alpha_V:=R_{-X-r_-}(0)U_-(V)$ satisfying 
\begin{equation}\label{defofA} 
\forall \eps>0,\,\, \alpha_V\in C^{\frac{\mu_{\min}}{\mu_{\max}}-\eps}(\mc{M}), \quad (-X-r_-)\alpha_V=U_-(V). \end{equation}
The operator $U_-$ is not (a priori) skew-adjoint with respect to the measure $dm$: one has $U_-^*=-U_--{\rm div}(U_-)$ where ${\rm div}(U_-)\in C^{1-}(\mc{M})$ is the divergence of $U_-$ with respect to the contact measure $dm$. We observe that 
\begin{equation}\label{casV=r_-}
V=r_- \Longrightarrow \alpha_V={\rm div}(U_-).
\end{equation}
Indeed, taking the adjoint of \eqref{XU_-}, we have the identity of operators
\[ (-X-r_-)U_-^*=U_-^*(-X+r_-)-r_-U_-^*\]  
and therefore 
\[ (-X-r_-)({\rm div}(U_-))=-(-X-r_-)U_-^*(1)=-U_-^*(r_-)+r_-U_-^*(1)=U_-(r_-).\]
which shows \eqref{casV=r_-}. In particular we see that $\alpha_V\in C^{1-}(\mc{M})$ in that case.  

Now we can give a short proof of the following 
\begin{theo}\label{mainresult}
Let $V\in C^{2-}(\mc{M})$, $W:=V-r_-$, $P:=-X+V$ and $P':=-X+W$. Let 
$\alpha_V$ be the function of \eqref{defofA} and assume that $\alpha_V\in C^{s-}(\mc{M})$ for some $s\in [\frac{\mu_{\min}}{\mu_{\max}},1)$.
In the region $\{{\rm Re}(\la)>{\rm Pr}(W-r_-)\}$, the operator $R_{P'}(\la)(U_-+\alpha_V): C^{\infty}(\mc{M})\to \mc{D}'(\mc{M})$ is analytic and one has  the identity
\begin{equation}\label{identresolv2}  
(U_-+\alpha_V)R_{P}(\la)=R_{P'}(\la)(U_-+\alpha_V).
\end{equation}
in $\{{\rm Re}(\la)>-\mu_{\min}s+{\rm Pr}(W)\}$. 
For each generalized resonant state $u$ of $P$ with resonance $\la_0$ contained in $\{{\rm Re}(\la)>-\mu_{\min}s+{\rm Pr}(W)\}$,  we have  $(U_-+\alpha_V)u=0$.
\end{theo}
\begin{proof}  
It suffices to prove \eqref{identresolv2} for ${\rm Re}(\la)$ large enough and then use meromorphic continuation in $\la$.  
Let $u\in C^\infty(\mc{M})$ and assume that ${\rm Re}(\la)>\mu_{\max}+V_{\max}$. 
By Lemma \ref{existunstable}, we have 
\[ [-X+V,U_-+\alpha_V]=r_-(U_-+\alpha_V)-U_-(V)-(X+r_-)(\alpha_V)=r_-(U_-+\alpha_V)
\]
and thus 
\begin{equation}\label{eqonU-v} 
\begin{gathered}
(-X+V-r_--\la)((U_-+\alpha_V)R_P(\la)u-R_{P'}(\la)(U_-+\alpha_V)u)=\\
(U_-+\alpha_V)(-X+V-\la)R_P(\la)u-(U_-+\alpha_V)u=0.
\end{gathered}\end{equation}
Thus $\omega:=(U_-+\alpha_V)R_P(\la)u-R_{P'}(\la)(U_-+\alpha_V)u$ is in $\ker (P'-\la)$. We also 
know from Corollary \ref{cor1} that  $(U_-+\alpha_V)R_P(\la)u\in C^0(\mc{M})$ and from Lemma \ref{resolvr}  that  $R_{P'}(\la)U_-u\in C^0(\mc{M})$. By Lemma \ref{resolvr} again, we know that there is no $C^0$ solution to $(P'-\la)\omega=0$ in $\{{\rm Re}(\la)> -\mu_{\min}+V_{\max}\}$ thus $\omega=0$ and the proof of  \eqref{identresolv2}  in $\{{\rm Re}(\la)>\mu_{\max}+V_{\max}\}$ is complete. Among the terms in \eqref{identresolv2}, all have meromorphic extension to $\{{\rm Re}(\la)>-\mu_{\min}s+{\rm Pr}(W)\}$ as operators mapping 
$C^\infty(\mc{M})$ to $\mc{D}'(M)$. 
Taking the residue at a resonance $\la_0\in \{{\rm Re}(\la)>-\mu_{\min}s+{\rm Pr}(W)\}$ in the identity \eqref{identresolv2}, we obtain 
\[ (U_-+\alpha_V) \Pi_{\la_0}=0,\]
if $\Pi_{\la_0}={\rm Res}_{\la_0}R_P(\la)$, thus the range of $\Pi_{\la_0}$ belongs to $\ker (U_-+\alpha_V)$, i.e generalized resonant states are in $\ker (U_-+\alpha_V)$. 
\end{proof}
We can view the first order differential operator $U_-+\alpha_V$ as a connection along the unstable leaves. There are three cases of particular interest which follow:
taking $V=0$ in the first case, $V=r_-$ in the second case and $V=\demi r_-$ in the third case, we obtain (using \eqref{casV=r_-})
\begin{corr}\label{maincor}
1) The operator $R_{-X-r_-}(\la)U_-: C^{\infty}(\mc{M})\to \mc{D}'(\mc{M})$ is analytic in the region $\{{\rm Re}(\la)>-\mu_{\min}\}$ and one has in that region
\begin{equation}\label{identresolv}  
U_-R_{-X}(\la)=R_{-X-r_-}(\la)U_-.
\end{equation}
Each generalized resonant state $u$ of $-X$ with resonance $\la_0$ contained in the region $\{{\rm Re}(\la)>-\mu_{\min}\}$ satisfies $U_-u=0$.\\
2) The operator $R_{-X}(\la)U_-^*: C^{\infty}(\mc{M})\to \mc{D}'(\mc{M})$ is analytic in the region $\{{\rm Re}(\la)>0\}$ and one has in $\{{\rm Re}(s)>h_{\rm top}-\mu_{\min}\}$
\begin{equation}\label{identresolvstar}  
U_-^*R_{-X+r_-}(\la)=R_{-X}(\la)U_-^*,
\end{equation} 
where $h_{\rm top}={\rm Pr}(0)$ is the topological entropy of the flow of $X$. Each generalized resonant state $u$ of $-X+r_-$ with resonance $\la_0$ contained in  $\{{\rm Re}(\la)>h_{\rm top} -\mu_{\min}\}$ satisfies $U_-^*u=0$.\\
3) The operator $R_{-X-\demi r_-}(\la)(U_-+\demi {\rm div}(U_-)): C^{\infty}(\mc{M})\to \mc{D}'(\mc{M})$ is analytic in the region $\{{\rm Re}(\la)>{\rm Pr}(-\demi r_-)\}$ and the following identity holds 
\begin{equation}\label{identresolvstar}  
(U_-+\demi {\rm div}(U_-))R_{-X+\demi r_-}(\la)=R_{-X-\demi r_-}(\la)(U_-+\demi {\rm div}(U_-)),
\end{equation} 
in $\{{\rm Re}(s)>{\rm Pr}(-\demi r_-)-\mu_{\min}\}$.
Each generalized resonant state $u$ of $-X+\demi r_-$ with resonance $\la_0$ contained in  $\{{\rm Re}(\la)>{\rm Pr}(-\demi r_-)-\mu_{\min}\}$ satisfies $(U_-+\demi {\rm div}(U_-))u=0$.
\end{corr}
The study of the spectrum in the third case, with potential $V=\demi r_-$, has been studied in details by Faure-Tsujii \cite{FaTs2} using the Grassmanian extension. It is particularly interesting since the first band of resonances concentrate near $\{{\rm Re}(\la)=0\}$. It can be noted that the horocyclic derivative $\mc{U}_-:=U_-+\demi {\rm div}(U_-)$ is skew-adjoint with respect to the contact measure $dm$.

\subsection{Second proof}
The second proof is more technical. For simplicity we only deal with the case $V=0$.
First, we need the following 
\begin{prop}\label{regulariteomega}
Let $u\in \mc{H}^{1-\eps}$ for some $\eps>0$ and define $\omega=U_-u$. There exist two pseudo-differential operators $A_1,A_2\in \Psi^0(\mc{M})$  such that ${\rm WF}(A_2)$ is contained in a small conic neighborhood $W$ of $E_u^*$, $A_1+A_2={\rm Id}$  and 
\[ A_1\omega \in L^2(\mc{M}) , \quad A_2\omega\in H^{-2+\eps}(\mc{M}).\]
\end{prop} 
\begin{proof}
Let $B_1,B_2\in \Psi^0(\mc{M})$ so that $B_1+B_2={\rm Id}$, $B_2$ is microsupported (i.e has wave-front set) in a small conic neighborhood $W$ of $E_u^*$ and $B_1$ is microsupported outside a small conic neighborhood $W'\subset W$ of $E_u^*$. Then, due to the property of $\mc{H}^{1-\eps}$ recalled in 3) of Proposition \ref{fauresjos}, $u=u_1+u_2$ with $u_1:= B_1u$ and $u_2=B_2u$ and $B_2u\in H^{-1+\eps}(\mc{M})$, $B_1u\in H^{N_0}(\mc{M})$ for some large $N_0\geq 1$. We obtain $\omega=
U_-u_1+U_-u_2$, with $u_1\in H^{N_0}(\mc{M})$ and $u_2\in H^{-1+\eps}(\mc{M})$.
Let $A_1,A_2\in \Psi^0(\mc{M})$  satisfying the same properties as $B_1,B_2$, then 
\[ \omega= A_1U_-u_1+A_1U_-u_2+A_2\omega \]
with  $A_2\omega \in H^{-2+\eps}(\mc{M})$  and $A_1U_-u_1\in L^2(\mc{M})$. The only term we need to analyse is $A_1U_-u_2$ and to show that it is in $L^2$. By using a partition of unity we can reduce to the case where  $u_2$ is supported in a small chart. Let 
us then consider in a small chart $\mc{O}$ near a point $x_0\in \mc{M}$ the distribution $A_1U_-u_2$: we can write $U_-=\sum_{j=1}^{3}a_j(x)\pl_{x_j}$ in a coordinate system $x=(x_1,x_2,x_3)$ where the chart becomes a neighborhood of $x=0$, with $a_j\in C^{2-}(\mc{M})$. 
We can also arrange the coordinate system so that $E_u^*=dx_1$ at $x=0$ 
and, since $E_u^*$ is a continuous bundle, so that $E_u^*\subset \mc{O}\x V$ over the chart $\mc{O}$, where $V\subset \rr^3$ is a small conic open neighborhood of $dx_1$ containing $W\cap \pi^{-1}(\mc{O})$ (here $\pi:T^*\mc{M}\to \mc{M}$ is the canonical projection). 
We have that $\pl_{x_i}u_2$ is microlocally $H^{N_0-1}$ outside $\mc{O}\x V$. Let $\chi$ be a smooth function on $\rr^3\setminus\{0\}$ which is homogeneous of degree $0$ and equal to $1$ in $V$ and $0$ outside a small conic neighborhood of $V$.
We can write 
$\pl_{x_j}u_2={\rm Op}(|\xi|^{2-\eps}\chi)u_2'+ r_2$ for some $u_2'\in L^2(\mc{O})$ and $r_2\in 
H^{N_0-1}(\mc{O})$. Now, we can use the paradifferential calculus of Bony \cite{Bo}, in particular Theorem 3.4 in \cite{Bo} shows that ${\rm Op}(a_j(x)|\xi|^{2-\eps}\chi(\xi))=T+R$ where $R:L^2(\mc{O})\to L^2(\mc{O})$ is bounded and $T$ is the paradifferential operator associated to the symbol $\sigma: (x,\xi)\mapsto a_j(x)|\xi|^{2-\eps}\chi(\xi)$ which belongs to $C^{2-}S_{\rm cl}^{2-\eps}(\mc{O})$. Using that  $\sigma$ vanishes outside a small conic neighbordhood $\mc{O}\x V'$ of $\mc{O}\x V$,  \cite[Corollary 3.5]{Bo} tells us that for each $v\in L^2$, $Tv$ is microlocally $L^2$ outside $V'$ in the sense that for each $Q\in \Psi^0(\mc{O})$ with microsupport not intersecting $\mc{O}\x V'$, $QTv\in L^2(\mc{O})$. 
Using this with $v=u_2'$, we obtain that $Q(a_j\pl_{x_j}u_2)\in L^2(\mc{O})$ and therefore
$A_1U_-u_2\in L^2(\mc{M})$. This concludes the proof.
\end{proof}
The main technical estimate is the following 
\begin{prop}\label{estimprop}
Let $\eps>0$ and $A\in \Psi^0(\mc{M})$ be a pseudo-differential operators such that 
${\rm WF}(A)$ is contained in a small conic neighborhood of $E_u^*$.
For all $\delta\in(0,\eps)$,  there exists $C_{\delta,\eps}>0$  such that for all $t\leq 0$, all $\omega\in H^{-2+\eps}(\mc{M})$ and all $f\in C^2(\mc{M})$
\begin{equation}\label{toprove}
\Big|\cjg e^{\int_0^t\mc{L}^s(r_--\mu_{\min}) ds}\mc{L}^tA\omega, f\cjd\Big|\leq C_{\delta,\eps}e^{8\delta |t|} ||\omega||_{H^{-2+\eps}(\mc{M})}||f||_{C^2(\mc{M})}.
\end{equation}
\end{prop}
\begin{proof} We fix $\mu_-<\mu_{\min}$ arbitrarily close to $\mu_{\min}$ and $\mu_+>\mu_{\max}$ arbitrarily close to $\mu_{\max}$.
We first write 
\begin{equation}\label{relationomega}
\cjg \mc{L}^tA\omega, 
e^{\int_0^t (\mc{L}^sr_--\mu_{-})ds}f\cjd=\cjg A\omega, 
e^{-\int_0^{-t}(\mc{L}^{s}r_--\mu_{-})ds}\mc{L}^{-t}f\cjd.\end{equation}
To simplify notations, we define 
\[ \hat{r}_-:=r_--\mu_{-}\]
which satisfies that there is a constant $C>0$ so that for each $z\in \mc{M}$ and $t\leq 0$
\begin{equation}\label{boundweight}
|e^{-\int_0^{-t}\mc{L}^s\hat{r}_-(z)ds}|\leq C.
\end{equation} 
By using a partition of unity, we reduce to the case where $A$ is supported in a small chart. Let 
$A':=A(1+\Delta_G)^{1-\eps/2}\in \Psi^{2-\eps}(\mc{M})$  with microsupport contained in a small conic neighborhood of $E_u^*$ and $\omega':=(1+\Delta_G)^{-1+\eps/2}\omega$. Note that 
$||\omega'||_{L^2}\leq ||\omega||_{H^{-2+\eps}}$.
 In local coordinates $x$ of the chart, we can write $A'={\rm Op}(a')$ where 
 $a'(x,\xi)$ is a smooth classical symbol of order $2-\eps$ satisfying 
\[ |\pl_x^{\alpha}\pl_\xi^\beta a'(x,\xi)|\leq C_{\alpha,\beta}\cjg \xi\cjd^{2-\eps-|\beta|}, 
\quad a'(x,\xi)\sim \sum_{j=0}^\infty a'_{j}(x,\xi)\]
with $a_{j}'$ homogeneous of degree $2-\eps-j$ in $\xi$. We also have that $a'(x,\xi)$ and its derivatives decay to infinite order in $\xi$ outside a small conic neighborhood $W\in T^*\mc{M}$ of $E_u^*$ (identifying $\mc{M}$ with $\rr^3$ via the chart).
Let $U_+$ be a local $C^{2-}$ section of $E_s$ in the chart ($U_+$ has the properties stated in Lemma \ref{existunstable}). Let $p(x,\xi)$ be the principal symbol of $U_+$ in the chart: it is in $C^{2-}_{\rm cl}S^1(\rr^3)$). Let $\chi\in C^\infty S_{\rm cl}^0(\rr^3)$ be a smooth symbol so that $\chi=1$ on $W$ and $\chi=0$ in a conic neighborhood of $p(x,\xi)=0$.
Let $b:=a'/(p^2+1)\in C^{2-}_{\rm cl}S^{-\eps}(\rr^3)$, which is decaying (with its derivatives) to 
infinite order outside $W$.  We write 
\[ U_+^2{\rm Op}(b)={\rm Op}(a')+U_+{\rm Op}(r_1)+{\rm Op}(r_2),\]
where $r_1,r_2$ are given by 
\[\begin{gathered} 
r_2= -\frac{a'}{p^2+1}+\frac{(\chi-1)}{p}U_+\Big(\frac{a'p}{p^2+1}\Big)-U_+\Big(\frac{\chi}{p}U_+\Big(\frac{a'p}{p^2+1}\Big)\Big), \\ 
r_1=U_+\Big(\frac{a'}{p^2+1}\Big)+\frac{\chi}{p}U_+\Big(\frac{a'p}{p^2+1}\Big),
\end{gathered}\]
and satisfy $r_1\in C^{1-}S_{\rm cl}^{-\eps}(\rr^3)$ and $r_2\in C^{0}S_{\rm cl}^{-\eps}(\rr^3)$. Here we have used that  $U_+^2(p)\in C^0$ due to Lemma \ref{existunstable}. 
By Lemma \ref{boundednessPSDO}, we have the following boundedness
\begin{equation}\label{Opprop}
\begin{gathered}
 {\rm Op}(b): L^2(\rr^3)\to H^{\eps'}(\rr^3), \quad {\rm Op}(r_1): L^2(\rr^3)\to L^2(\rr^3), \\ {\rm Op}(r_2): L^2(\rr^3)\to L^2(\rr^3).
\end{gathered}\end{equation}
for each $\eps'<\eps$. The adjoint of $U_+$ for the invariant measure $\alpha\wedge d\alpha$ 
is $U_+^*=-U_+-{\rm div}(U_+)$, with ${\rm div}(U_+)\in C^{1-}(\rr^3)$. We can then write 
\begin{equation}\label{pairing}
\begin{split}
\cjg \omega,e^{-\int_0^{-t}\mc{L}^s\hat{r}_-\, ds}\mc{L}^{-t}f\cjd= & 
\cjg {\rm Op}(b)\omega', (U_+^*)^2e^{-\int_0^{-t}\mc{L}^s\hat{r}_-\, ds}\mc{L}^{-t}f\cjd\\
&- \cjg {\rm Op}(r_1)\omega', U_+^*e^{-\int_0^{-t}\mc{L}^s\hat{r}_-\, ds}\mc{L}^{-t}f\cjd\\
&- \cjg {\rm Op}(r_2)\omega', e^{-\int_0^{-t}\mc{L}^s\hat{r}_-\, ds}\mc{L}^{-t}f\cjd.
\end{split}\end{equation}
Here, the first term involving $(U_+^*)^2$ and the second term involving $(U_+)^*$ makes sense for the following reason: since 
\[ (U_+^*)^2=U_+^2+2{\rm div}(U_+)U_++({\rm div}(U_+))^2+U_+({\rm div}(U_+)),\]
the sum of the 3 first terms gives a second order differential operator with $C^{1-}$ coefficients, the last term is the multiplication operator by the function 
$U_+({\rm div}(U_+))\in \cap_{\eps>0} H^{-\eps}(\rr^3)$. Moreover, the function 
\[F_t:=e^{-\int_0^{-t} 
\mc{L}^s\hat{r}_- ds}\mc{L}^{-t}f\in C^{2-}(\rr^3),\]
thus 
\[ \begin{gathered}
U_+({\rm div}(U_+))F_t\in \cap_{\eps>0} H^{-\eps}(\rr^3),\quad 
U_+^2(F_t)\in \cap_{\eps>0} H^{-\eps}(\rr^3),\\
{\rm div}(U_+)U_+(F_t)\in C^{1-}(\rr^3),\quad 
(U_+)^*(F_t)\in C^{1-}(\rr^3).
\end{gathered}
\]
Using \eqref{Opprop} and $\omega'\in L^2$, all the pairings in \eqref{pairing} make sense.
Let us now estimate the terms in \eqref{pairing} with respect to $t$. First, we have 
\begin{equation}\label{bound1}
\Big|\cjg {\rm Op}(r_2)\omega', F_t\cjd\Big|\leq C||\omega'||_{L^2}||f||_{C^0}.
\end{equation}
Since $U_+$ belongs to $E_s$,  we have $|d\varphi_{-t}.U_+|=\mc{O}(e^{-\mu_{-}|t|})$ for $t\leq 0$, thus
\[ \begin{split}
\Big|U_+(F_t)\Big|\leq &
 \Big| \int_{0}^{-t} {d\hat{r}_-}_{\varphi_s}.d\varphi_s U_+ ds\Big| ||f||_{C^0}+ \Big| df_{\varphi_{-t}}.d\varphi_{-t} U_+\Big|\\
 \leq & C||f||_{C^0} +Ce^{-\mu_{-}|t|}||f||_{C^1}
\end{split}\]
(here and later $C$ depends on $\mu_-$). This implies the bound 
\begin{equation}\label{bound2}
\Big|\cjg {\rm Op}(r_1)\omega', U_+^*F_t\cjd\Big|\leq C||\omega'||_{L^2}( ||f||_{C^0} +Ce^{-\mu_{-}|t|}||f||_{C^1}).
\end{equation}
It remains to analyse the first term in \eqref{pairing}.  Similarly as above, one has 
\begin{equation}\label{est1} 
\Big|2{\rm div}(U_+)U_+(F_t)+({\rm div}(U_+))^2F_t\Big|\leq C||f||_{C^0} +Ce^{-\mu_{-}|t|}||f||_{C^1}
\end{equation}
and ${\rm Op}(b)\omega'$ paired with that term is bounded like \eqref{bound2}. 
Next, we can use the bilinear estimate, for each $\delta>0$,
 $||fu||_{H^\delta}\leq C_\delta ||f||_{C^{2\delta}}||u||_{H^{\delta}}$ 
for some $C_\delta>0$, to deduce that  
\begin{equation}\label{est2} 
|| U_+({\rm div}(U_+)).F_t ||_{H^{-\delta}(\mc{M})}\leq C_{\delta} ||U_+({\rm div}(U_+))||_{H^{-\delta}(\mc{M})} ||F_t||_{C^{2\delta}(\mc{M})}\end{equation}
and using interpolation estimates between $C^0$ and $C^1$ norm of $F_t$, we have 
\begin{equation}\label{est3} 
||F_t||_{C^{2\delta}(\mc{M})}\leq C||F_t||^{1-2\delta}_{C^0}||F_t||^{2\delta}_{C^{1}}\leq 
C ||f||_{C^{1}}e^{2\delta \mu_{+}|t|}.
\end{equation}
To deal with $U_+^2(F_t)$, we first rewrite for $t\leq 0$
\[  F_t= \mc{L}^{|t|}(f)e^{-\int_{0}^{|t|} \mc{L}^s(\hat{r}_-)ds}\]
and use the identity of operators (following from Lemma \ref{existunstable}) for $t\in\rr$
\[ U_+\mc{L}^{t}=
 e^{-\int_{0}^{t} \mc{L}^s(r_+)ds}\mc{L}^{t}U_+.\]
We get for $t\leq 0$
\[ U_+(F_t)=e^{-\int_{0}^{|t|} \mc{L}^s(\hat{r}_-)ds}\Big(e^{-\int_0^{|t|}\mc{L}^s(r_+)ds}\mc{L}^{|t|}U_+f-\int_0^{|t|}e^{-\int_0^{s}\mc{L}^u(r_+)du}\mc{L}^s(U_+\hat{r}_-)ds\, \mc{L}^{|t|}f
\Big)\]
and reapplying $U_+$, this gives 
\[\begin{split}
U_+^2(F_t)=e^{-\int_{0}^{|t|} \mc{L}^s(\hat{r}_-)ds}\Big[ & \Big( \int_0^{|t|}e^{-\int_0^s\mc{L}^u(r_+)du}\mc{L}^s(U_+\hat{r}_-)ds
\Big)^2\mc{L}^{|t|}f+ e^{-2\int_0^{|t|}\mc{L}^s(r_+)ds}\mc{L}^{|t|}(U_+^2f)\\
& -2 \Big(\int_0^{|t|}e^{-\int_0^s\mc{L}^u(r_+)du}\mc{L}^s(U_+\hat{r}_-)ds\Big)e^{-\int_0^{|t|}\mc{L}^s(r_+)ds}\mc{L}^{|t|}(U_+f)\\
& - \mc{L}^{|t|}(U_+f)e^{-\int_0^{|t|}\mc{L}^s(r_+)ds}\int_0^{|t|}e^{-\int_0^s\mc{L}^u(r_+)du}\mc{L}^s(U_+r_+)ds\\
& +\Big(\int_0^{|t|}e^{-\int_0^s\mc{L}^u(r_+)du}\mc{L}^s(U_+\hat{r}_-)\int_0^se^{-\int_0^u\mc{L}^v(r_+)dv}\mc{L}^u(U_+r_+)du\,
ds\\
& \quad \quad -\int_0^{|t|}e^{-2\int_0^s\mc{L}^u(r_+)du}\mc{L}^s(U_+^2\hat{r}_-)ds
\Big)\mc{L}^{|t|}f\Big].
\end{split}\]
Using that $r_\pm\in C^{2-}(\mc{M})$ and that for $s\geq 0$, $r=r_+$ or $r=\hat{r}_-$ 
we have 
\[ \Big|e^{-\int_0^s \mc{L}^u(r)du}\Big|\leq C\]
for some $C$ independent of $s$, we see that the four first lines of the identity giving $U_+^2(F_t)$ are bounded in $C^0$ norm by  
$Ct^{2}||f||_{C^2}$ for $t\leq 0$. The only term that remains to be analysed is the $H^{-\delta}(\mc{M})$ norm of the distribution $W_t\mc{L}^{|t|}f$ where
\[ W_t:=-e^{-\int_{0}^{|t|} \mc{L}^s(\hat{r}_-)ds}\int_0^{|t|}e^{-2\int_0^s\mc{L}^u(r_+)du}\mc{L}^s(U_+^2\hat{r}_-)ds.\]
By Lemma \ref{normL^s}, we have for all $t\leq 0$
\[\begin{split} 
||W_t||_{H^{-\delta}}\leq &  C|t|\,||e^{-\int_{0}^{|t|} \mc{L}^s(\hat{r}_-)ds}||_{C^{2\delta}}\sup_{s\in [0,|t|]}(||e^{-2\int_0^s\mc{L}^u(r_+)du}||_{C^{2\delta}}e^{\delta|s|\mu_{+}})
||U_+^2\hat{r}_-||_{H^{-\delta}}\\
\leq &  Ce^{6\delta |t|\mu_{+}} ||U_+^2\hat{r}_-||_{H^{-\delta}}
\end{split} \]
where, as above, we have used interpolation between $C^0$ and $C^1$ to bound the 
$C^{2\delta}$ norms of the terms $e^{-\int_{0}^{|t|} \mc{L}^s(r_\pm)ds}$. Now we get 
\[ ||W_t\mc{L}^{|t|}f||_{H^{-\delta}}\leq C_\delta e^{6\delta |t|\mu_{+}} ||U_+^2\hat{r}_-||_{H^{-\delta}} ||\mc{L}^{|t|}f||_{C^{2\delta}}\leq C_\delta e^{8\delta |t|\mu_{+}} ||U_+^2\hat{r}_-||_{H^{-\delta}}||f||_{C^1}.\]
We conclude 
that for all $\delta>0$  there exists $C_\delta>0$  and all $t\leq 0$
\begin{equation}\label{sobolev2eps} 
||U_+^2(F_t)||_{H^{-\delta}(\mc{M})}\leq C_\delta e^{8\delta |t|\mu_{+}}||f||_{C^2}.
\end{equation}
Combining \eqref{est1}, \eqref{est2}, \eqref{est3} with \eqref{sobolev2eps}, we obtain that for each $\delta\in (0,\eps)$ there is $C_{\delta,\eps}$ depending on $\delta$ and $\eps$ such that and $t\leq 0$
\[ 
\cjg {\rm Op}(b)\omega', (U_+^*)^2F_t\cjd\leq C_{\delta,\eps} e^{8\delta |t|\mu_{+}}||\omega'||_{L^2}||f||_{C^2}.\]
Combining this with \eqref{bound1} and \eqref{bound2}, we get our final estimate 
\[ |\cjg \omega,e^{-\int_0^{-t}\mc{L}^s\hat{r}_-\, ds}\mc{L}^{|t|}f\cjd |\leq 
C_{\delta,\eps} e^{8\delta  |t|\mu_{+}}||\omega'||_{L^2}||f||_{C^2}\leq 
C_{\delta,\eps} e^{8\delta |t|\mu_{+}}||\omega||_{H^{-2+\eps}}||f||_{C^2}\]
which shows \eqref{toprove}. 
\end{proof}

For $N>0$, we use the notation $\mc{H}^{N}$ for the anisotropic Sobolev space of Proposition \ref{fauresjos} with $N_0\gg 1$ very large, and we let $(C^2(\mc{M}))'$ be the dual Banach space of $C^2(\mc{M})$.
\begin{theo}\label{mainresult2}
Let $\eps>0$, the operator $R_{-X-r_-}(\la)U_-: \mc{H}^{1-\eps}\to (C^2(\mc{M}))'$ is an analytic family of bounded operators in the region ${\rm Re}(\la)>-\mu_{\min}(1-\eps)$. 
Let ${\rm Re}(\la)>-\mu_{\min}(1-\eps)$ and let $\omega\in H^{-2+\eps}(\mc{M})$ be such that there exist $A_1,A_2\in \Psi^0(\mc{M})$ with $A_1+A_2={\rm Id}$ such that ${\rm WF}(A_2)$ is contained in a small conic neighborhood of $E_u^*$, 
$A_1\omega \in L^2(\mc{M})$ and $A_2\omega\in H^{-2+\eps}(\mc{M})$,
then 
\begin{equation}\label{nokernel}
(-X-r_--\la)\omega=0\Rightarrow \omega=0.
\end{equation} 
As a consequence, if $u$ is a generalized resonant state of $-X$ with resonance $\la_0$ in the region ${\rm Re}(\la)>-\mu_{\min}$, then $U_-u=0$.
\end{theo}
\begin{proof} Let ${\rm Re}(\la)>-\mu_{\min}(1-\eps)$. 
To prove that $R_{-X-r_-}(\la)U_-u$ is analytic for $u\in  \mc{H}^{1-\eps}$, we use Proposition \ref{estimprop} with $A_2U_-u$: take $\delta>0$ small enough so that ${\rm Re}(\la)+\mu_{\min}>8\delta$, then for each $f\in C^2(\mc{M})$ 
\[ \int_{-\infty}^0\Big |\cjg e^{\la t+\int_{0}^t \mc{L}^sr_-ds}\mc{L}^t(A_2U_-u), f\cjd \Big| dt 
\leq C_{\delta,\eps}||U_-u||_{H^{-2+\eps}}||f||_{C^2}\]
and the trivial inequality 
\[\int_{-\infty}^0\Big |\cjg e^{\la t+\int_{0}^t \mc{L}^sr_-ds}\mc{L}^t(A_1U_-u), f\cjd \Big| dt 
\leq C_{\delta,\eps}||A_1U_-u||_{L^2}||f||_{L^2}\]
thus $R_{-X-r_-}(\la)U_-:\mc{H}^{1-\eps}\to (C^2(\mc{M}))'$ is analytic.

Let us show \eqref{nokernel}.   We set $\omega(t):=e^{tX}\omega=\mc{L}^t\omega$, then in the weak sense
\[\pl_t\omega(t)=-e^{tX}(r_-+\la)\omega=-\omega(t)(\mc{L}^tr_-+\la)\]
and thus 
\begin{equation}\label{omegat}
\omega(t)=\omega e^{-\la t-\int_0^t \mc{L}^sr_- ds}.
\end{equation}
We proceed by contradiction: assume that there is $f\in C^\infty(\mc{M})$ such that $\cjg \omega,f\cjd\not=0$. Let $\mu_-=\mu_{\min}(1-\eps/2)$. Using \eqref{omegat},  we write for $t\leq 0$
\begin{equation}\label{relationomega}
e^{-\la  t-\mu_{-}t}\cjg \omega,f\cjd=\cjg \mc{L}^t\omega, 
e^{\int_0^t (\mc{L}^sr_--\mu_{-})ds}f\cjd=\cjg \omega, 
e^{-\int_0^{-t}(\mc{L}^{s}r_--\mu_{-})ds}\mc{L}^{-t}f\cjd.
\end{equation}
By the estimate \eqref{toprove},  we can take $\delta>0$ small enough so that ${\rm Re}(\la)+\mu_{-}>8\delta$ and we let $t\to -\infty$ in \eqref{relationomega}, and we obtain a contradiction to $\cjg \omega,f\cjd\not=0$, which shows that $\omega=0$.

If $u$ is a resonant state with resonance $\la_0$ and ${\rm Re}(\la_0)>-\mu_{\min}(1-\eps)$, then $u\in \mc{H}^{1-\eps}(\mc{M})$ and by Lemma \ref{existunstable}, we get for $\omega:=U_-u$
\[ 0=U_-(-X-\la_0)u=(-X-\la_0-r_-)\omega.\]
According to Proposition \ref{regulariteomega}, $\omega$ has the sufficient property to apply \eqref{nokernel}, thus $\omega=0$.
If now $u$ is a generalized resonant states with resonance $\la_0$, it satisfies $(-X-\la_0)^ju=u_0$ for some resonant state $u_0$ and some $j\in \nn$. Using an induction assumption that the generalized resonant state $(-X-\la_0)u=(-X-\la_0)^{j-1}u_0$ is in $\ker U_-$, we get $(-X-\la_0-r_-)U_-u=U_-(-X-\la_0)^{j-1}u_0=0$
and we can aplpy \eqref{nokernel} to $\omega=U_-u$.
\end{proof}

\subsection{Applications: invariant distributions for $U_-$ and obstruction to solutions of the cohomological equation}\label{lastsec}

We recall the result of Faure-Tsujii \cite{FaTs} describing the localisation of Ruelle resonances. 
For a potential $V\in C^{\infty}(\mc{M})$ let us define the quantities for $k=0,1$
\[\begin{gathered}
\gamma_k^+:=\lim_{t\to +\infty}\sup_{z\in \mc{M}} \frac{1}{t}\int_0^t (V-(\demi+k)r_-)\circ \varphi_s(z) ds,\\
\gamma_k^-:=\lim_{t\to +\infty}\inf_{z\in \mc{M}} \frac{1}{t}\int_0^t (V-(\demi+k)r_-)\circ \varphi_s(z) ds.
\end{gathered}\]
In particular, when $V=0$, this gives
\[ \gamma_0^+=-\frac{1}{2} \mu_{\min} , \,\, \gamma_0^-=-\frac{1}{2} \mu_{\max}, \,\,
 \gamma_1^+=-\frac{3}{2}\mu_{\min}, \,\, \gamma_1^-=-\frac{3}{2}\mu_{\max}. \]
\begin{theo}[Faure-Tsujii \cite{FaTs}]\label{Fauretsujii}
Let $\mc{M}$ be a $3$-dimensional oriented manifold and let $X$ be a smooth vector field generating a contact Anosov flow and $V$ be a smooth potential. Then for each $\eps>0$ small, there exists only finitely many resonances of $P=-X+V$ in the region 
\[ \{{\rm Re}(\la)>\gamma_1^++\eps\}\setminus \{{\rm Re}(\la)\in [\gamma_0^--\eps,\gamma_0^++\eps]\}.\]
If $\gamma_1^+<\gamma_0^-$, then there is 
infinitely many resonances in $\{{\rm Re}(\la)\in [\gamma_0^--\eps,\gamma_0^++\eps]\}$, with a Weyl type asymptotics. In the case $V=0$, the condition $\gamma_1^+<\gamma_0^-$ can be rewritten as the pinching condition $3\mu_{\min}>\mu_{\max}$.
\end{theo}
The proof of Corollary \ref{cormain} about the existence of infinitely many distributions in the Sobolev space $H^{-\demi\frac{\mu_{\max}+\eps}{\mu_{\min}}}(\mc{M})$ that are horocyclic invariant (for all $\eps>0$) is a direct consequence of Theorems \ref{main} and \ref{Fauretsujii}, applied with $V=0$.\\

In \cite{FlFo}, the analysis of the distributions in $\ker U_-$ allows in constant curvature to solve the cohomological equation $U_-f=g$ for $g$ with a given regularity.
In our case, the operator $U_-^*\not=-U_-$ in general.
To say something about the cohomological equation for $U_-$ in certain spaces, one has to know something about the kernel of $U_-^*$. In particular, using Corollary \ref{maincor},  the generalised resonant states of $-X+r_-$ in $\{{\rm Re}(\la)>h_{\rm top}-\mu_{\min}\}$ are elements in $\ker U_-^*$ inside $(C^{q}(\mc{M}))'$ for each $q<1$, and thus provides obstructions to solve $U_-f=g$ with $f\in C^2(\mc{M}),g\in C^1(\mc{M})$: to have a solution of $U_-f=g$ with $f\in C^2(\mc{M})$, $g$ must satisfy $\cjg u, g\cjd=0$ for all generalised resonant states $u$ of $-X+r_-$ with resonances $\la_0$ 
such that ${\rm Re}(\la_0)>h_{\rm top}-\mu_{\min}$.

\end{document}